\documentclass[a4paper,a4wide,10pt]{article} 
\usepackage[utf8]{inputenc}
\usepackage{amsmath,amsthm, amssymb}
\usepackage{lmodern}
\usepackage[usenames,dvipsnames]{pstricks}
\usepackage{enumitem} 
\usepackage{chngcntr}
\usepackage{dsfont}
\usepackage{geometry}
\usepackage{color}
\usepackage{mathtools}
\usepackage{graphicx}
\newtheorem{theorem}{Theorem}[section]
\newtheorem{lemma}[theorem]{Lemma}
\newtheorem{prop}[theorem]{Proposition}
\newtheorem{problem}[theorem]{Problem}
\theoremstyle{definition}
\newtheorem{defn}[theorem]{Definition}
\newtheorem{cor}[theorem]{Corollary}
\newtheorem{rem}[theorem]{Remark}

\DeclareMathOperator*\cp{Cap}

\DeclareMathOperator*\hdim{dim_{H}}
\DeclareMathOperator*\s{\Sigma}
\DeclareMathOperator*\Om{\Omega}
\DeclareMathOperator*\diam{diam}

\newcommand*\diff{\mathop{}\, d}
\newcommand\numberthis{\addtocounter{equation}{1}\tag{\theequation}}

\numberwithin{equation}{section}
\numberwithin{figure}{section}

\author{Bohdan Bulanyi\footnote{LJLL UMR 7598, Universit\'e de Paris, France. e-mail: bulanyi@math.univ-paris-diderot.fr} }

\date{\today}

\begin{document}
\font\myfont=cmr12 at 18pt
\title{\myfont On the importance of the connectedness assumption in the statement of the optimal $p$-compliance problem}
\maketitle

\begin{abstract}
In this short note we provide a proof of the importance of the connectedness assumption in the statement of the optimal $p$-compliance problem with length penalization and in the statement of the constrained form of this problem for the existence of solutions.
\end{abstract}

\section{Introduction}

A spatial dimension $N\geq 2$ and an exponent $p \in (1,+\infty)$ are given. Let $\Omega$ be an open bounded set in $\mathbb{R}^{N}$ and let $f$ belong to $L^{q_{0}}(\Om)$ with 
\begin{equation} \label{Eq 1.1}
q_{0}=(p^{*})^{\prime}\,\ \text{if}\,\ 1<p<N, \qquad q_{0}>1 \,\ \text{if} \,\ p=N, \qquad q_{0}=1 \,\ \text{if} \,\ p>N,
\end{equation}
where $p^{*}=Np/(N-p)$ and $(1/p^{*})+(1/(p^{*})^{\prime})=1$. In this paper, every open set is nonempty.

In view of the Sobolev embeddings (see, for instance, \cite[Theorem 7.10]{PDE}), the functional  $E_{f,\Omega}$ defined over $W^{1,p}_{0}(\Omega)$ by
\[
 E_{f,\Omega}(u)=\frac{1}{p}\int_{\Omega} |\nabla u|^{p}\diff x - \int_{\Omega} fu \diff x
\]
is finite.  It is a well-known fact that for each closed proper subset $\Sigma$ of $\overline{\Omega}$ the functional $E_{f,\Omega}$ admits a unique minimizer $u_{f,\Omega,\Sigma}$ over $W^{1,p}_{0}(\Omega\backslash \Sigma)$. Furthermore, the Sobolev function $u_{f,\Omega,\Sigma}$ is a unique solution to the Dirichlet problem
\begin{equation*} 
-\Delta_{p}u = f \,\ \text{in}\,\ \Omega\backslash \Sigma,\,\ u \in W^{1,p}_{0}(\Omega\backslash \Sigma). 
\end{equation*}
The latter means that $u_{f,\Omega,\Sigma} \in W^{1,p}_{0}(\Omega\backslash \Sigma)$ and
\begin{equation*} 
\int_{\Om} |\nabla u_{f,\Omega,\Sigma}|^{p-2} \nabla u_{f,\Om,\Sigma} \nabla \varphi \diff x= \int_{\Om}f\varphi \diff x\qquad  \forall \varphi \in W^{1,p}_{0}(\Omega\backslash \Sigma).
\end{equation*}
Notice that if a closed set $\Sigma \subset \overline{\Omega}$ has zero $p$-capacity (for the definition of capacity, see Section 2), then $u_{f,\Omega,\Sigma}=u_{f,\Omega,\emptyset}$ (see Remark~\ref{rem 2.8}). The dependence of $u_{f,\Omega,\Sigma}$ on $p$ is neglected in this paper. For each closed proper subset $\Sigma$ of $\overline{\Omega}$, the $p$-compliance functional at $\Sigma$ is defined by
\[
C_{f,\Omega}(\Sigma)=-E_{f,\Omega}(u_{f,\Omega,\Sigma})=\frac{1}{p^{\prime}}\int_{\Omega}|\nabla u_{f,\Omega,\Sigma}|^{p}\diff x = \frac{1}{p^{\prime}}\int_{\Omega}fu_{f,\Omega,\Sigma}\diff x.
\]

The purpose of this paper is to prove the importance of the connectedness assumption in the statements of the following two shape optimization problems for the existence of solutions to these problems. 
\begin{problem} \label{P 1.1}
	Let $p \in (N-1, +\infty)$. Given $\lambda>0$, find a set $\s \subset \overline{\Om}$ minimizing the functional $\mathcal{F}_{\lambda,f,\Omega}$ defined by
	\[
	\mathcal{F}_{\lambda,f,\Omega}({\s}^{\prime})=C_{f,\Omega}({\s}^{\prime})+ \lambda \mathcal{H}^{1}({\s}^{\prime})
	\]
	among all sets $\Sigma^{\prime}$ in the class $\mathcal{K}(\Omega)$ of all closed  connected proper subsets of $\overline{\Omega}$.
\end{problem}

\begin{problem} \label{P 1.2}
	Let $p \in (N-1, +\infty)$. Given $L>0$, find a set $\s \subset \overline{\Om}$ minimizing the $p$-compliance functional $C_{f,\Omega}$ among all sets $\Sigma^{\prime}$ in the class  $\mathcal{A}_{L}(\Omega)$ of all closed connected subsets of $\overline{\Omega}$ satisfying the constraint $ 0<\mathcal{H}^{1}(\Sigma^{\prime})\leq L.$ 
\end{problem}
These two problems were studied in \cite{Butazzo-Santambrogio,MR3063566,MR3195349,Opt,p-compl,partialregularity}, however, until this paper, the importance of the connectedness of admissible sets in their statements for the existence of solutions had not been proved, but it was simply mentioned as a remark without proof in \cite{Opt} for the special case $N=p=2$. In this paper, we provide a detailed proof in any spatial dimension $N\geq 2$ for every $p\in (N-1,+\infty)$ and for the sharp integrability assumption on the source term $f \in L^{q_{0}}(\Omega)$. 
 
It is worth noting that any closed set $\Sigma^{\prime}\subset\overline{\Omega}$ with $\mathcal{H}^{1}(\Sigma^{\prime})<+\infty$ is removable for the Sobolev space $W^{1,p}_{0}(\Omega)$ if $p\in (1,N-1]$ (see, for instance, Theorem~\ref{thm 2.2} and Remark~\ref{rem 2.8}), namely, $W^{1,p}_{0}(\Omega \backslash \Sigma^{\prime})=W^{1,p}_{0}(\Omega)$ and this implies that $C_{f,\Omega}(\Sigma^{\prime})=C_{f,\Omega}(\emptyset)$. Thus, defining Problem~\ref{P 1.1} for some exponent $p \in (1, N-1]$, we would get only trivial solutions to this problem: every point $x_{0}$ in $\overline{\Omega}$ and the empty set. On the other hand, every set $\Sigma^{\prime}$ in $\mathcal{A}_{L}(\Omega)$ would be a solution to Problem~\ref{P 1.2} if this problem were defined for $p\in (1,N-1]$. Also, if $\Sigma^{\prime}\subset \overline{\Omega}$ is a closed set such that $\Sigma^{\prime}\cap \Omega$ is of Hausdorff dimension one and with finite $\mathcal{H}^{1}$-measure, then $\Sigma^{\prime}$ is not removable for $W^{1,p}_{0}(\Omega)$ if and only if $p\in (N-1,+\infty)$ (see Corollary~\ref{cor 2.4} and Remark~\ref{rem 2.8}). Furthermore, a point $x_{0}\in \Omega$ is not removable for $W^{1,p}_{0}(\Omega)$ if and only if $p\in (N,+\infty)$ (see Theorem~\ref{thm 2.2}, Remark~\ref{rem 2.3} and Remark~\ref{rem 2.8}). Therefore, Problem~\ref{P 1.1} and Problem~\ref{P 1.2} are interesting only in the case when $p\in (N-1,+\infty)$.

We also assume that $f \neq 0$ in $L^{q_{0}}(\Omega)$, because otherwise the $p$-compliance functional $C_{f,\Omega}(\cdot)$ would be reduced to zero, and then each solution to Problem~\ref{P 1.1} would be either a point $x_{0} \in \overline{\Omega}$ or the empty set,  and any set in $\mathcal{A}_{L}(\Omega)$ would be a solution to Problem~\ref{P 1.2}. 

It is clear that replacing the class $\mathcal{K}(\Omega)$ in the statement of Problem~\ref{P 1.1} by the class of all closed connected subsets $\Sigma^{\prime}$ of $\overline{\Omega}$ satisfying the estimate $C_{f,\Omega}(\Sigma^{\prime})+\lambda\mathcal{H}^{1}(\Sigma^{\prime})\leq C_{f,\Omega}(\emptyset)$ does not actually change Problem~\ref{P 1.1}. Thus, from now on we shall talk only about those admissible sets $\Sigma^{\prime}$ for Problem~\ref{P 1.1} for which the estimate $C_{f,\Omega}(\Sigma^{\prime})+\lambda\mathcal{H}^{1}(\Sigma^{\prime})\leq C_{f,\Omega}(\emptyset)$ holds. 

The connectedness assumption, together with the bounds on the length (i.e., $\mathcal{H}^{1}$-measure) of admissible sets for Problem~\ref{P 1.1} and Problem~\ref{P 1.2}, gives the necessary compactness to prove that these two problems admit a solution (see \cite{sverak} for $p=2$ and \cite{Bucur} for the general case).

In this paper, we prove the following two theorems.
\begin{theorem} \label{thm 1.3} Let $\Om \subset \mathbb{R}^{N}$ be open and bounded, $\lambda>0$, $p \in (N-1, + \infty)$, $f \in L^{q_{0}}(\Om)$, $f \not =0$,~$q_{0}$~is defined in (\ref{Eq 1.1}). Then the existence of minimizers for the functional $\mathcal{F}_{\lambda,f,\Omega}$ over the class of all closed proper subsets of $\overline{\Omega}$ fails.
\end{theorem}
\begin{theorem} \label{thm 1.4} Let $\Om \subset \mathbb{R}^{N}$ be open and bounded, $L>0$, $p \in (N-1, + \infty)$, $f \in L^{q_{0}}(\Om)$, $f \not =0$,~$q_{0}$~is defined in (\ref{Eq 1.1}). Then the existence of minimizers for the $p$-compliance functional $C_{f,\Omega}$ over the class $\bigl\{\Sigma^{\prime}\subset \overline{\Omega}: \,\ \Sigma^{\prime}\,\ \text{is closed},\,\ 0<\mathcal{H}^{1}(\Sigma^{\prime})\leq L\bigr\}$ fails.
\end{theorem}
\section{Preliminaries}
 For any nonempty open set $U\subset \mathbb{R}^{N}$, $W^{1,p}(U)$ is the Sobolev space of functions defined on $U$ whose distributional gradient $\nabla u$ belongs to $L^{p}(U;\mathbb{R}^{N})$. We denote by $W^{1,p}_{0}(U)$ the closure of $C^{\infty}_{0}(U)$ in $W^{1,p}(U)$, where $C^{\infty}_{0}(U)$ is the space of functions in $C^{\infty}(U)$ with compact support in $U$. By $C^{0}(X)$ we shall denote the space of continuous functions on $X$. For each set $A \subset \mathbb{R}^{N}$, the set $A^{c}$ will denote its complement in $\mathbb{R}^{N}$, that is, $A^{c}=\mathbb{R}^{N}\backslash A$. We shall denote by $\mathcal{H}^{d}(A)$ the $d$-dimensional Hausdorff measure of $A$. 

Let us recall the notion of the Bessel capacity (see e.g. \cite{Potential}, \cite{Ziemer}). 
\begin{defn} \label{def 2.1} \textit{ For $p\in (1,+\infty)$, the Bessel $(1,p)$-capacity of a set $E\subset \mathbb{R}^{N}$ is defined as
		\[
		{\rm Cap}_{p}(E)=\inf \bigl\{\|f\|^{p}_{L^{p}(\mathbb{R}^{N})} :\, g*f \geq 1\,\ \text{on}\,\ E,\,\ f\in L^{p}(\mathbb{R}^{N}),\,\ f \geq 0\bigr\},
		\]
		where the Bessel kernel $g$ is defined as that function whose Fourier transform is}
	\[
	\hat{g}(\xi)=(2\pi)^{-\frac{N}{2}}\bigl(1+|\xi|^{2}\bigr)^{-\frac{1}{2}}.
	\]
\end{defn}
We say that a property holds $p$-quasi everywhere (abbreviated as $p$-q.e.) if it holds except on a set $A$ where $\cp_{p}(A)=0$. It is worth mentioning that, by \cite[Corollary 2.6.8]{Potential}, for every $p\in (1,+\infty)$, the notion of the Bessel capacity ${\rm Cap}_{p}$ is equivalent to the following 
\[
\widetilde{{\rm Cap}_{p}}(E)=\inf_{u \in W^{1,p}(\mathbb{R}^{N})}\biggl\{\int_{\mathbb{R}^{N}}|\nabla u|^{p}\diff x + \int_{\mathbb{R}^{N}}|u|^{p}\diff x : u \geq 1 \,\ \text{on some neighborhood of $E$}\biggr\}
\]
in the sense that there exists $C=C(N,p)>0$ such that for any set $E\subset \mathbb{R}^{N}$,
\[
\frac{1}{C}\widetilde{{\rm Cap}_{p}}(E) \leq {\rm Cap}_{p}(E) \leq C\widetilde{{\rm Cap}_{p}}(E).
\]
The notion of capacity is crucial in the investigation of the pointwise behavior of Sobolev functions.

For convenience, we recall the next theorems and propositions.

\begin{theorem}\label{thm 2.2} Let $E\subset \mathbb{R}^{N}$ and $p\in (1,N]$. Then $\cp_{p}(E)=0$ if $\mathcal{H}^{N-p}(E)<+\infty$. Conversely, if $\cp_{p}(E)=0$, then $\mathcal{H}^{N- p+\varepsilon}(E)=0$ for every $\varepsilon>0$.
\end{theorem}
\begin{proof} For a proof of the fact that  ${\rm Cap}_p(E)=0$ if $\mathcal{H}^{N-p}(E)<+\infty$, we refer to \cite[Theorem 5.1.9]{Potential}. The fact that if ${\rm Cap}_{p}(E)=0,$ then $\mathcal{H}^{N-p+\varepsilon}(E)=0$ for every $\varepsilon>0$ follows from \cite[Theorem 5.1.13]{Potential}. 
\end{proof}

\begin{rem}\label{rem 2.3} Let $p \in (N,+\infty)$. Then, there exists $C=C(N,p)>0$ such that for any nonempty set $E\subset \mathbb{R}^{N}$, ${\rm Cap}_p(E)\geq C$. We can take $C=  {\rm Cap}_{p}(\{0\})$, which is positive by \cite[Proposition 2.6.1 (a)]{Potential}, and use the fact that the Bessel $(1,p)$-capacity is an invariant under translations and is nondecreasing with respect to set inclusion.
\end{rem}

Recall that for all $E\subset \mathbb{R}^{N}$ the number
\[
\hdim(E)=\sup\bigl\{s \in [0,+\infty): \mathcal{H}^{s}(E)=+\infty\bigr\}=\inf\bigl\{t\in [0,+\infty): \mathcal{H}^{t}(E)=0\bigr\}
\]
is called the Hausdorff dimension of $E$.

\begin{cor}\label{cor 2.4} Let $E \subset \mathbb{R}^{N}$, $\hdim(E)=1$ and $\mathcal{H}^{1}(E)<+\infty$. Then ${\rm Cap}_p(E)>0$ if and only if $p\in (N-1, +\infty)$.
\end{cor}
\begin{proof}[Proof of Corollary \ref{cor 2.4}] If $p>N$, then by Remark~\ref{rem 2.3}, ${\rm Cap}_{p}(E)>0$. Assume by contradiction that ${\rm Cap}_{p}(E)=0$ for some $p \in (N-1,N]$. Taking $\varepsilon= (p-N+1)/2$  so that $0<N-p+\varepsilon<1$, by Theorem~\ref{thm 2.2} we get, $\mathcal{H}^{N-p+\varepsilon}(E)=0$, but this leads to a contradiction with the fact that $\hdim(E)=1$. On the other hand, if $p \in (1, N-1]$, then $\mathcal{H}^{N-p}(E)<+\infty$ and by Theorem~\ref{thm 2.2}, ${\rm Cap}_{p}(E)=0$. This completes the proof of Corollary~\ref{cor 2.4}.
\end{proof}

\begin{prop} \label{prop 2.5} Let $t\in (0,1]$ and $A_{t}=[0, t]\times \{0\}^{N-1}$. Then the following assertions hold.
	\begin{enumerate}[label=(\roman*)]
		\item If $p \in (N-1, N)$, then there exists a constant $C=C(N,p)>0$ such that
		\[
		t^{N-p} \leq C {\rm Cap}_{p}(A_{t}).
		\]
		\item If $p=N$, then there exists a constant $C=C(N)>0$ such that
		\[
		\biggl(\log\biggl(\frac{C}{t}\biggr)\biggr)^{1-p} \leq C {\rm Cap}_{p}(A_{t}).
		\]
	\end{enumerate}
\end{prop}
\begin{proof} Since $\diam(A_{t})\leq 1$, $(i)$ and $(ii)$ follows from \cite[Corollary 5.1.14]{Potential}.
\end{proof}
\begin{defn} \label{def 2.6} Let the function $u$ be defined $p$-q.e. on $\mathbb{R}^{N}$ or on some open subset. Then $u$ is said to be $p$-quasi continuous if for every $\varepsilon>0$ there is an open set $A$ with ${\rm Cap}_{p}(A)<\varepsilon$ such that the restriction of $u$ to the complement of $A$ is continuous in the induced topology.
\end{defn}

\begin{theorem} \label{thm 2.7} Let $Y\subset \mathbb{R}^{N}$ be an open set and $p \in (1,+\infty)$. Then for each $u \in W^{1,p}(Y)$ there exists a $p$-quasi continuous function $\widetilde{u} \in W^{1,p}(Y)$, which is uniquely defined up to a set of ${\rm Cap}_{p}$-capacity zero and $u=\widetilde{u}$ a.e. in $Y$.
\end{theorem}
\begin{proof} We refer the reader, for instance, to the proof of \cite[Theorem 2.8]{p-compl}, which actually applies for the general spatial dimension $N\geq 2$.
\end{proof}

\begin{rem} \label{rem 2.8} A Sobolev function $u \in W^{1,p}(\mathbb{R}^{N})$ belongs to $W^{1,p}_{0}(Y)$ if and only if its $p$-quasi continuous representative $\widetilde{u}$ vanishes $p$-q.e. on $Y^{c}$ (see \cite[Theorem 4]{BAGBY} and \cite[Lemma 4]{Hedberg}). Thus, if $Y^{\prime}$ is an open subset of $Y$ and $u \in W^{1,p}_{0}(Y)$ such that $\widetilde{u}=0$ $p$-q.e. on $Y\backslash Y^{\prime}$, then the restriction of $u$ to $ Y^{\prime}$ belongs to $W^{1,p}_{0}(Y^{\prime})$ and conversely, if we extend a function $u \in W^{1,p}_{0}(Y^{\prime})$ by zero on $Y\backslash Y^{\prime}$, then $u \in W^{1,p}_{0}(Y)$. It is worth mentioning that if $\Sigma \subset \overline{Y}$ and ${\rm Cap}_{p}(\s)=0$, then $W^{1,p}_{0}(Y) = W^{1,p}_{0}(Y \backslash \s)$. Indeed, $u \in W^{1,p}_{0}(Y)$ if and only if $u \in W^{1,p}(\mathbb{R}^{N})$ and $\widetilde{u}=0$ $p$-q.e. on $Y^{c}$ that is equivalent to say $u\in W^{1,p}(\mathbb{R}^{N})$ and $\widetilde{u}=0$ $p$-q.e. on $Y^{c} \cup \s$ (since ${\rm Cap}_{p}(\Sigma)=0$ and ${\rm Cap}_{p}(\cdot)$ is a subadditive set function, see \cite[Proposition 2.3.6]{Potential}) or $u \in W^{1,p}_{0}(Y \backslash \s)$. In the sequel we shall always identify $u \in W^{1,p}(Y)$ with its $p$-quasi continuous representative $\widetilde{u}$.
\end{rem}

\begin{prop} \label{prop 2.9}  Let $D\subset \mathbb{R}^{N}$ be a bounded extension domain and let $u \in W^{1,p}(D)$. Consider the set $E=\overline{D} \cap \{x: u(x)=0\}$. If $\cp_{p}(E)>0$, then there exists a constant $C=C(N,p,D)>0$ such that 
	\[
	\int_{D}|u|^{p}\diff x \leq C ({\rm Cap}_p(E))^{-1} \int_{D} |\nabla u|^{p}\diff x.
	\]
\end{prop}
\begin{proof} For a proof, see, for instance, \cite[Corollary 4.5.3, p. 195]{Ziemer}.
	
\end{proof}

\section{Proofs of Theorem~\ref{thm 1.3} and Theorem~\ref{thm 1.4}}
To prove Theorem~\ref{thm 1.3} and Theorem~\ref{thm 1.4}, we need the following two lemmas.
\begin{lemma}\label{lem 3.1} Let $p\in (N-1, +\infty)$, $a \in (0,1),\, \delta>0$ and $u \in W^{1,p}((0,\delta)^{N})$ satisfying $u=0\,\ p$-q.e. on $\bigl[\delta/2-a \delta/2, \delta/2+ a \delta/2\bigr] \times \bigl\{\delta/2\bigr\}^{N-1}$. Then there exists $C=C(N,p)>0$ such that
	\[
	\int_{(0, \delta)^{N}} |u|^{p} \diff x \leq C \delta^{p} ({\rm Cap}_{p}([0, a]\times \{0\}^{N-1}))^{-1} \int_{(0, \delta)^{N}} |\nabla u|^{p}\diff x.
	\]
\end{lemma}
\begin{proof}Define a function $v \in W^{1,p}((0,1)^{N})$ by $v(\cdot)=u(\delta(\cdot))$. Then we observe that $v=0$ $p$-q.e. on $\bigl[1/2-a/2,1/2+a/2\bigr]\times~\bigl\{1/2\bigr\}^{N-1}$. By Proposition~\ref{prop 2.9} and the fact that the Bessel $(1,p)$-capacity is an invariant under translations and is nondecreasing with respect to set inclusion, there exists a constant $C=C(N,p)>0$ such that 
	\[
	\int_{(0,1)^{N}} |v|^{p} \diff y \leq C ({\rm Cap}_{p}([0, a]\times \{0\}^{N-1}))^{-1} \int_{(0,1)^{N}} |\nabla v|^{p}\diff y.
	\]
	Finally, using the change of variables $x=\delta y$, we recover the desired estimate.
\end{proof}

\begin{lemma}\label{lem 3.2} Let $\Omega\subset \mathbb{R}^{N}$ be open and bounded, $\Sigma$ be a closed proper subset of $\overline{\Omega}$, $p \in (1,+\infty)$ and $f_{1}, f_{2} \in L^{q_{0}}(\Omega)$, where $q_{0}$ is defined in (\ref{Eq 1.1}). Let $z:[0,+\infty)\to [0, +\infty)$ be defined by
	\[
	z(t)=t^{p^{\prime}}\,\ \, \text{if} \,\ 2\leq p<+\infty, \qquad z(t)=\Bigl(\|f_{1}\|^{p^{\prime}}_{L^{q_{0}}(\Omega)}+\|f_{2}\|^{p^{\prime}}_{L^{q_{0}}(\Omega)}\Bigr)^{2-p}t^{p}\,\ \, \text{if}\,\  1<p<2.
	\]
	Then there exists a constant $A=A(p)>0$ such that
	\[
	C_{f_{1},\Omega}(\Sigma)\leq 2^{p-1}C_{f_{2},\Omega}(\Sigma)+ Az(\|f_{1}-f_{2}\|_{L^{q_{0}}(\Omega)}).
	\]
\end{lemma}
\begin{proof} According to \cite[Theorem 2.3]{ABC}, there exists $C=C(p)>0$ such that
	\[
	\int_{\Omega}|\nabla u_{f_{1},\Omega,\Sigma}-\nabla u_{f_{2},\Omega, \Sigma}|^{p}\diff x \leq Cz(\|f_{1}-f_{2}\|_{L^{q_{0}}(\Omega)}).
	\]
	Since for any nonnegative numbers $c$ and $d$, $(c+d)^{p}\leq 2^{p-1}(c^{p}+d^{p})$, we deduce that 
	\begin{align*}
	\frac{1}{p^{\prime}}\int_{\Omega}|\nabla u_{f_{1},\Omega,\Sigma}|^{p}\diff x &\leq \frac{2^{p-1}}{p^{\prime}}\int_{\Omega}|\nabla u_{f_{2},\Omega,\Sigma}|^{p}\diff x + \frac{2^{p-1}}{p^{\prime}}\int_{\Omega}|\nabla u_{f_{1},\Omega, \Sigma}-\nabla u_{f_{2}, \Omega, \Sigma}|^{p}\diff x\\
	&\leq \frac{2^{p-1}}{p^{\prime}}\int_{\Omega}|\nabla u_{f_{2},\Omega,\Sigma}|^{p}\diff x+ 2^{p-1}Cz(\|f_{1}-f_{2}\|_{L^{q_{0}}(\Omega)}).
	\end{align*}
	Thus, defining $A=2^{p-1}C$, we complete the proof of Lemma~\ref{lem 3.2}.
\end{proof}
\begin{proof}[Proof of Theorem~\ref{thm 1.3}] Since $\Omega$ is bounded, there exists $R>0$ such that $\Omega$ is contained in the $N$-cube $Q=(-R,R)^{N}$.\\
	\textit{Step 1.} We start by proving that for any $g \in L^{p^{\prime}}(Q)$, $\inf \bigl\{ \mathcal{F}_{\lambda,g,\Omega}(\Sigma): \Sigma \subsetneqq \overline{\Omega}\,\ \text{is closed}\bigr\}=0.$
	First of all, notice that $\inf\bigl\{\mathcal{F}_{\lambda, g, \Omega}(\Sigma): \Sigma\subsetneqq \overline{\Omega}\,\ \text{is closed}\bigr\}$ is equal to 
	\begin{equation}\label{3.1}
	\inf \biggl\{ \frac{1}{p^{\prime}} \int_{\Omega} |\sigma|^{p^{\prime}} \diff x +  \lambda \mathcal{H}^{1}(\Sigma): \Sigma \subsetneqq \overline{\Omega} \,\ \text{is closed,}\,\ \sigma \in L^{p^{\prime}}(\Omega; \mathbb{R}^{N}), -div(\sigma)=g \,\ \text{in} \,\ \mathcal{D}^{\prime}(\Omega\backslash \Sigma)\biggr\},
	\end{equation}
	which is a direct consequence of \cite[Lemma A.3]{p-compl}. Fix an arbitrary $\varepsilon \in (0,1)$. We construct a sequence $\{(\sigma_{n},S_{n})\}_{n\in \mathbb{N}^{*}}$ of admissible pairs for problem~(\ref{3.1}) such that
	\[
	\limsup_{n\to+\infty}\biggl(\frac{1}{p^{\prime}}\int_{\Omega} |\sigma_{n}|^{p^{\prime}}\diff x+\lambda \mathcal{H}^{1}(S_{n})\biggr) \leq \lambda 2^{N}R\varepsilon.
	\]
	For each integer $n\geq 1$ and each point $\xi^{n}$ in the set~$\bigl\{jR/n : j \in \bigl\{-n,...,n-1 \bigr\} \bigr\}^{N}$, we define the open ``local $N$-cube" $Q(\xi^{n})\subset Q$ by
	\[
	Q(\xi^{n})=\Biggl(\xi^{n}+\biggl(0,\frac{R}{n}\biggr)^{N}\Biggr),
	\]
	the ``crack'' set $S(\xi^{n})\subset Q(\xi^{n})$ by
	\[
	S(\xi^{n})=\Biggl(\xi^{n}+\biggl[\frac{R}{2n} -\frac{\varepsilon R}{2 n^{N}}, \frac{R}{2n} + \frac{\varepsilon R}{2 n^{N}}\biggr] \times \biggl\{\frac{R}{2n}\biggr\}^{N-1}\Biggr)
	\]
	(see Figure \ref{Figure grid}), and the space $W_{\xi^{n}}$ consisting of the Sobolev functions $w \in W^{1,p}(Q(\xi^{n}))$ vanishing $p$-q.e. on $S(\xi^{n})$, that is,
	\[
	W_{\xi^{n}}=\bigl\{w \in W^{1,p}(Q(\xi^{n})) : w=0 \,\ p \text{-q.e. on}\;\ S(\xi^{n})\bigr\}.
	\] 
	\begin{figure}
		\centering
		\includegraphics[width=.5\textwidth]{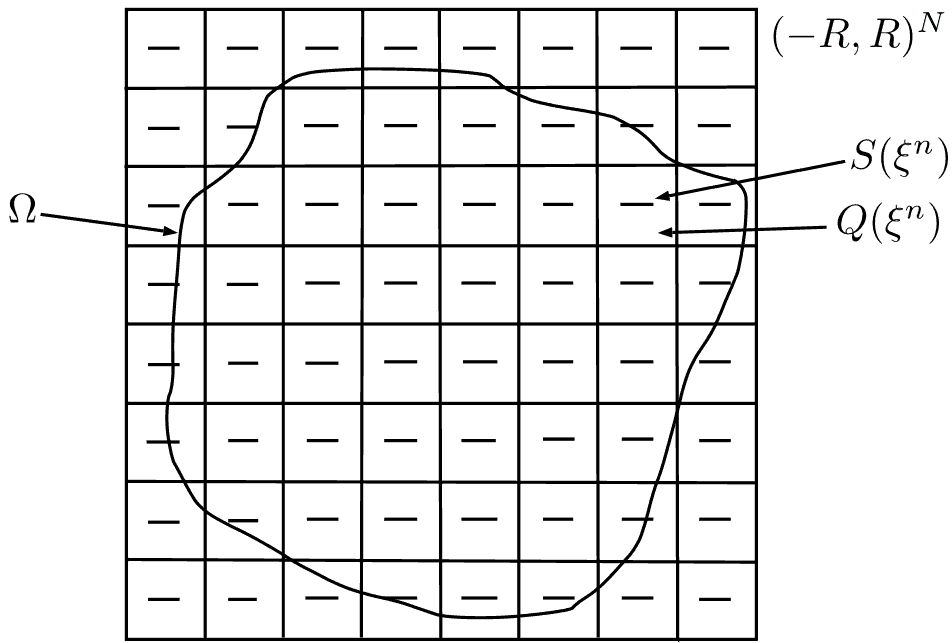}
		\caption{Construction of the $N$-cubes $Q(\xi^{n})$ and the crack sets $S(\xi^{n})$ in the proof of Theorem~\ref{thm 1.3}.}
		\label{Figure grid}
	\end{figure}
	Notice that $W_{\xi^{n}}$ is closed in $W^{1,p}(Q(\xi^{n}))$. Indeed, assume that $(w_{m})_{m\in \mathbb{N}}\subset W_{\xi^{n}}$ and $w_{m}\to w$ in $W^{1,p}(Q(\xi^{n}))$. We fix a function $\chi \in C^{\infty}_{0}(Q(\xi^{n}))$ such that $\chi=1$ on $S(\xi^{n})$. For each $m\in \mathbb{N},$ we have $w_{m}\chi \in W^{1,p}(Q(\xi^{n}))$ and $w_{m}\chi =0$ $p$-q.e. on $(Q(\xi^{n}))^{c}\cup S(\xi^{n})$. Then, according to Remark~\ref{rem 2.8}, $(w_{m}\chi)_{m\in \mathbb{N}}\subset W^{1,p}_{0}(Q(\xi^{n})\backslash S(\xi^{n}))$. In addition, $w_{m}\chi \to w \chi$ in $W^{1,p}_{0}(Q(\xi^{n})\backslash S(\xi^{n}))$ and hence $w\chi=0$ $p$-q.e. on $(Q(\xi^{n}))^{c}\cup S(\xi^{n})$ (see Remark~\ref{rem 2.8}). But this implies that $w=0$ $p$-q.e. on $S(\xi^{n})$ and therefore $w \in W_{\xi_{n}}$. For convenience, we define the set  
	\[
	A_{\varepsilon,n}=\biggl[0, \frac{\varepsilon}{n^{N-1}}\biggr]\times \biggl\{0\biggr\}^{N-1}.
	\]
	Next, changing variables if necessary and applying Lemma~\ref{lem 3.1} with $a=\varepsilon/n^{N-1}$ and $\delta=R/n$, we deduce that for each $w \in W_{\xi^{n}}$, the inequality
	\[
	\int_{Q(\xi^{n})} |w|^{p}\diff x \leq \frac{C}{n^{p}} ({\rm Cap}_{p}(A_{\varepsilon,n}))^{-1} \int_{Q(\xi^{n})} |\nabla w|^{p}\diff x \numberthis \label{3.2}
	\]
	holds for some $C=C(N,p,R)>0$. Since $(W_{\xi^{n}},\|\cdot\|_{W^{1,p}(Q(\xi^{n}))})$ is a reflexive Banach space and the functional $F_{\xi^{n}}:W_{\xi^{n}}\to \mathbb{R}$ defined by
	\[
	F_{\xi^{n}}(w)=\frac{1}{p}\int_{Q(\xi^{n})} |\nabla w|^{p}\diff x - \int_{Q(\xi^{n})} gw \diff x
	\]
	is lower semicontinuous, coercive (thanks to (\ref{3.2})) and strictly convex, using the direct method in the Calculus of Variations, we deduce that $F_{\xi^{n}}$ admits a unique minimizer $u_{\xi^{n}}$ over $W_{\xi^{n}}$. It should be noted that the same is true in the case when $g \in L^{q_{0}}(Q)$ thanks to the Sobolev embeddings, but we shall use, in particular, the fact that $g \in 
	L^{p^{\prime}}(Q)$ to derive some nice estimates below. It follows from the minimality of $u_{\xi^{n}}$ that 
	\[
	\lim_{t\to 0+}\frac{1}{t}(F_{\xi^{n}}(u_{\xi^{n}}+tw)-F_{\xi^{n}}(u_{\xi^{n}}))\geq 0 \,\ \,\ \text{and} \,\ \,\ \lim_{t\to 0+}\frac{1}{t}(F_{\xi^{n}}(u_{\xi^{n}}-tw)-F_{\xi^{n}}(u_{\xi^{n}}))\geq 0  \,\ \,\ \forall w \in W_{\xi^{n}}.
	\] This implies that 
	\[
	\int_{Q(\xi^{n})} |\nabla u_{\xi^{n}}|^{p-2} \nabla u_{\xi^{n}} \nabla w \diff x = \int_{Q(\xi^{n})} g w \diff x \qquad \forall w \in W_{\xi^{n}} \numberthis \label{3.3}
	\]
	 and, in particular,
	\[
	\int_{Q(\xi^{n})} |\nabla u_{\xi^{n}}|^{p} \diff x =  \int_{Q(\xi^{n})} g u_{\xi^{n}} \diff x. 
	\]
	Applying H\"{o}lder's inequality to the right-hand side of the latter formula and then the inequality (\ref{3.2}) to $u_{\xi^{n}}$, we obtain that
	\[
	\int_{Q(\xi^{n})} |\nabla u_{\xi^{n}}|^{p}\diff x\leq \frac{C^{\frac{1}{p}}}{n}({\rm Cap}_{p}(A_{\varepsilon,n}))^{-\frac{1}{p}}\Biggl(\int_{Q(\xi^{n})}|g|^{p^{\prime}}\diff x \Biggr)^{\frac{1}{p^{\prime}}}\Biggl(\int_{Q(\xi^{n})}|\nabla u_{\xi^{n}}|^{p}\diff x \Biggr)^{\frac{1}{p}}
	\]
	and hence
	\begin{equation}\label{3.4}
	\int_{Q(\xi^{n})} |\nabla u_{\xi^{n}}|^{p} \diff x  \leq \frac{\widetilde{C}}{n^{p^{\prime}}} ({\rm Cap}_{p}(A_{\varepsilon,n}))^{1-p^{\prime}} \int_{Q(\xi^{n})} |g|^{p^{\prime}} \diff x,
	\end{equation}
	where $\widetilde{C}=\widetilde{C}(N,p,R)>0$. Here we have used the fact that $g \in L^{p^{\prime}}(Q)$. For each $n\in\mathbb{N}^{*}$, let $\sigma_{n} \in L^{p^{\prime}}(Q; \mathbb{R}^{N})$ be defined as follows
	\[ 
	\left.\sigma_{n} \right |_{Q(\xi^{n})}:= |\nabla u_{\xi^{n}}|^{p-2} \nabla u_{\xi^{n}} \,\ \text{for each}\,\ Q(\xi^{n}) \subset Q.
	\]
	Also, for each $n\in \mathbb{N}^{*}$, define the compact set ${\s}_{n} \subset Q$ by $\Sigma_{n}:= \bigcup S(\xi^{n}),$ where the union is taken over all $S(\xi^{n})$ in $Q$. Then 
	\begin{align*}
	\int_{Q} |\sigma_{n}|^{p^{\prime}}\diff x  = \sum \int_{Q(\xi^{n})} |\nabla u_{\xi^{n}}|^{p}\diff x 
	&  \stackrel{(\ref{3.4})}{\leq} \frac{\widetilde{C}}{n^{p^{\prime}}} ({\rm Cap}_{p}(A_{\varepsilon,n}))^{1-p^{\prime}} \sum \int_{Q(\xi^{n})} |g|^{p^{\prime}}\diff x  \\
	& = \frac{\widetilde{C}}{n^{p^{\prime}}} ({\rm Cap}_{p}(A_{\varepsilon,n}))^{1-p^{\prime}} \int_{Q}|g|^{p^{\prime}}\diff x, \numberthis \label{3.5}
	\end{align*}
	where the summations are taken over all $N$-cubes $Q(\xi^{n}) \subset Q$. On the other hand,
	\[
	\mathcal{H}^{1}(\Sigma_{n})= 2^{N}R \varepsilon. \numberthis \label{3.6}
	\]
	Let us now fix an arbitrary function $\varphi \in C^{\infty}_{0}(\Omega\backslash {\s}_{n})$. We extend $\varphi$ by zero outside $\Omega\backslash \Sigma_{n}$ and keep the same notation for this extension. It is clear that the restriction of $\varphi$ on each $Q(\xi^{n})\subset Q$ belongs to $W_{\xi^{n}}$. Then, using the optimality condition (\ref{3.3}) and the fact that $\varphi=0$ on $\Omega^{c}\cup \Sigma_{n}$, we get
	\begin{align*}
	\int_{\Omega} \sigma_{n} \nabla \varphi \diff x = \sum \int_{Q(\xi^{n})} |\nabla u_{\xi^{n}}|^{p-2} \nabla u_{\xi^{n}} \nabla \varphi \diff x =  \sum \int_{Q(\xi^{n})} g  \varphi \diff x =  \int_{\Omega} g \varphi \diff x,
	\end{align*}
	where the summations are taken over all $N$-cubes $Q(\xi^{n}) \subset Q$. This implies that
	\[
	-div(\sigma_{n})=g \,\ \text{in}\,\ \mathcal{D^{\prime}}(\Omega\backslash {\s}_{n}).
	\]
	So for each $n\in\mathbb{N}^{*}$, defining $S_{n}$ by $S_{n}=\Sigma_{n}\cap\overline{\Omega}$, we observe that $(\sigma_{n}, S_{n})$ is an admissible pair for problem (\ref{3.1}). Let us consider the next three cases.\\
	\textit{Case 1: $p\in (N-1,N)$}. By Proposition~\ref{prop 2.5} $(i)$ applied with $t=\varepsilon/n^{N-1}$, there exists $C_{0}=C_{0}(N,p)>0$ such that
	\[
	\frac{\varepsilon^{N-p}}{n^{(N-1)(N-p)}} \leq C_{0} {\rm Cap}_{p}(A_{\varepsilon,n})
	\]
	and hence
	\[
	({\rm Cap}_{p}(A_{\varepsilon,n}))^{1-p^{\prime}} \leq C n^{(N-1)(N-p)(p^{\prime}-1)}
	\]
	for some $C=C(\varepsilon, p, N)>0$. Since $p\in (N-1,N)$, we observe that $(N-1)(N-p)(p^{\prime}-1)<p^{\prime}$.\\
	\textit{Case 2: $p=N.$} By Proposition~\ref{prop 2.5} $(ii)$ applied with $t=\varepsilon/n^{N-1}$, there exists $C_{0}=C_{0}(N)>0$ such that
	\[
	\biggr(\log\biggr(\frac{C_{0}n^{N-1}}{\varepsilon}\biggl)\biggl)^{1-p} \leq C_{0} {\rm Cap}_{p}(A_{\varepsilon,n})
	\]
	and hence
	\[
	({\rm Cap}_{p}(A_{\varepsilon,n}))^{1-p^{\prime}} \leq C \log (Cn)
	\]
	for some $C=C(\varepsilon, N)>0$.\\
	\textit{Case 3: $p>N$.} In this case, by Remark~\ref{rem 2.3}, there exists $C=C(N,p)>0$ such that 
	\[
	({\rm Cap}_{p}(A_{\varepsilon,n}))^{1-p^{\prime}} \leq C.
	\]
	Thus, returning to the estimate (\ref{3.5}), we can now conclude that for any fixed $p>N-1$ there exists a nonnegative function $\psi$, defined on $(0, +\infty)$, such that $\psi(n) \to 0$ as $n\to +\infty$ and
	\[
	\int_{Q} |\sigma_{n}|^{p^{\prime}}\diff x \leq \psi(n).
	\]
	Letting $n$ tend to $+\infty$ in the above estimate, taking into account the fact that $\mathcal{H}^{1}(S_{n})\leq \mathcal{H}^{1}(\Sigma_{n})$ and the estimate $(\ref{3.6})$, we get that the infimum in problem (\ref{3.1}) is less than or equal to $\lambda2^{N}R\varepsilon$ and hence
	\[
	\inf \bigl\{\mathcal{F}_{\lambda,g,\Omega}(\s): \Sigma \subsetneqq \overline{\Omega}\,\ \text{is closed}\bigr\} \leq \lambda 2^{N} R \varepsilon.
	\]
	As $\varepsilon \in (0,1)$ was arbitrarily chosen,
	\[
	\inf \bigl\{\mathcal{F}_{\lambda,g,\Omega}(\s): \Sigma \subsetneqq \overline{\Omega}\,\ \text{is closed}\bigr\}=0.
	\]
	\textit{Step 2.} We prove that $\inf\bigl\{\mathcal{F}_{\lambda,f,\Omega}(\Sigma): \s \subsetneqq \overline{\Omega}\,\ \text{is closed}\bigr\}=0$. Let us fix a sequence $(f_{m})_{m\in\mathbb{N}}\subset L^{p^{\prime}}(Q)$ such that $f_{m}\to f$ in $L^{q_{0}}(\Omega)$.
	Next, using Lemma~\ref{lem 3.2} and if $p\in (1,2)$, then using also the fact that the sequence $(\|f_{m}\|_{L^{q_{0}}(\Omega)})_{m\in \mathbb{N}}$ is bounded, we deduce that there exist a constant $A=A(p)>0$ and a nonnegative function $z\in C^{0}([0,+\infty))$ satisfying $z(0)=0$ such that for all $m\in \mathbb{N}$,
	\[
	\inf_{\Sigma \subsetneqq \overline{\Omega}\, \text{is closed}}\mathcal{F}_{\lambda,f,\Omega}(\Sigma)\leq 2^{p-1}\inf_{\Sigma \subsetneqq \overline{\Omega}\, \text{is closed}}\mathcal{F}_{\lambda,f_{m},\Omega}(\Sigma)+Az(\|f-f_{m}\|_{L^{q_{0}}(\Om)})=Az(\|f-f_{m}\|_{L^{q_{0}}(\Om)}),
	\]
	where we have used the result of \textit{Step 1}. Next, letting $m$ tend to $+\infty$ in the above estimate, we deduce that $\inf\bigl\{\mathcal{F}_{\lambda,f,\Omega}(\Sigma): \s \subsetneqq \overline{\Omega}\,\ \text{is closed}\bigr\}=0$. \\
	\textit{Step 3.} Assume by contradiction that there is a solution $\widetilde{\Sigma}$ to the problem \[\inf\bigl\{\mathcal{F}_{\lambda,f,\Omega}(\Sigma): \s \subsetneqq \overline{\Omega} \,\ \text{is closed}\bigr\}.\] From \textit{Step 2} it follows that $C_{f,\Omega}(\widetilde{\Sigma})=\mathcal{H}^{1}(\widetilde{\Sigma})=0$. Then $u_{f,\Omega,\widetilde{\Sigma}}=0$ as an element of $W^{1,p}(\Omega)$. By the minimality of $u_{f,\Omega,\widetilde{\Sigma}}$ (recall that $u_{f,\Omega,\widetilde{\Sigma}}$ is a unique minimizer of $E_{f,\Omega}$ over $W^{1,p}_{0}(\Omega\backslash \widetilde{\Sigma})$) and the fact that $u_{f,\Omega,\widetilde{\Sigma}}=0$,  
	\[
      0=\lim_{t\to 0+}\frac{1}{t}\Bigl(E_{f,\Omega}\Bigl(u_{f,\Omega,\widetilde{\Sigma}}+t\zeta\Bigr)-E_{f,\Omega}\Bigl(u_{f,\Omega,\widetilde{\Sigma}}\Bigr)\Bigr)=\lim_{t\to 0+}\frac{1}{t}E_{f,\Omega}(t\zeta)=-\int_{\Omega} f \zeta \diff x\,\ \,\ \forall \zeta \in C^{\infty}_{0}(\Omega \backslash \widetilde{\Sigma}),
	\]
	which implies that $f=0$ a.e. in $\Omega$ and leads to a contradiction. This completes our proof of Theorem~\ref{thm 1.3}.
\end{proof}
\begin{proof}[Proof of Theorem \ref{thm 1.4}] There exists $R>0$ such that $\Omega \subset (-R,R)^{N}$. Fix an arbitrary $g \in L^{p^{\prime}}((-R,R)^{N})$. This is a direct consequence of \cite[Lemma A.3]{p-compl} that $\inf\bigl\{C_{g,\Omega}(\Sigma): \Sigma \subset\overline{\Omega}\,\ \text{is closed}, \,\ 0<\mathcal{H}^{1}(\Sigma)\leq L\bigr\}$ is equal to
	\begin{equation*}
	\inf\biggl\{\frac{1}{p^{\prime}}\int_{\Omega}|\sigma_{\Sigma}|^{p^{\prime}}\diff x: \Sigma \subset \overline{\Omega}\,\ \text{is closed},\,\ 0<\mathcal{H}^{1}(\Sigma)\leq L,\,\ \sigma_{\Sigma}\in L^{p^{\prime}}(\Omega;\mathbb{R}^{N}),-div(\sigma_{\Sigma})=g\,\ \text{in}\,\ \mathcal{D}^{\prime}(\Omega\backslash \Sigma) \biggr\}. 
	\end{equation*}
Then, proceeding in the same way as in \textit{Step 1} in the proof of Theorem~\ref{thm 1.3}, we can construct a sequence $(\sigma_{\Sigma_{n}})_{n\in \mathbb{N}^{*}} \subset L^{p^{\prime}}(\Omega; \mathbb{R}^{N})$ such that for each $n\in \mathbb{N}^{*}$, $\Sigma_{n}\subset \overline{\Omega}$ is closed, $0<\mathcal{H}^{1}(\Sigma_{n})\leq L$, $-div(\sigma_{\Sigma_{n}})=g$ in $\mathcal{D}^{\prime}(\Omega\backslash \Sigma_{n})$ and, in addition,
\[
\lim_{n\to +\infty}\int_{\Omega}|\sigma_{\Sigma_{n}}|^{p^{\prime}}\diff x=0.
\]
Thus $\inf\bigl\{C_{g,\Omega}(\Sigma): \Sigma \subset\overline{\Omega}\,\ \text{is closed},\,\ 0<\mathcal{H}^{1}(\Sigma)\leq L\bigr\}=0.$ Now let $(f_{m})_{m\in\mathbb{N}}\subset L^{p^{\prime}}((-R,R)^{N})$ be a sequence such that $f_{m}\to f \,\ \text{in}\,\ L^{q_{0}}(\Omega)$. We already know that for each $m\in\mathbb{N}$, \[\inf\bigl\{C_{f_{m},\Omega}(\Sigma): \Sigma \subset\overline{\Omega}\,\ \text{is closed},\,\ 0<\mathcal{H}^{1}(\Sigma)\leq L\bigr\}=0.\] This, together with Lemma~\ref{lem 3.2} and the fact that $\|f-f_{m}\|_{L^{q_{0}}(\Omega)}\to 0$ as $m\to +\infty$, implies that 
\[
\inf\bigl\{C_{f,\Omega}(\Sigma): \Sigma \subset\overline{\Omega}\,\ \text{is closed}, \,\ 0<\mathcal{H}^{1}(\Sigma)\leq L\bigr\}=0.
\]
Suppose now by contradiction that there is a solution $\widetilde{\Sigma}$ to the above problem. Since $C_{f,\Omega} (\widetilde{\Sigma})=0$, we have $u_{f,\Omega,\widetilde{\Sigma}}=0$ as an element of $W^{1,p}(\Omega)$. Then, using the minimality of $u_{f,\Omega,\widetilde{\Sigma}}$, we deduce that 
\[
\int_{\Omega}f\zeta\diff x=0 \qquad \forall \zeta \in C^{\infty}_{0}(\Omega\backslash \widetilde{\Sigma}),
\]
which implies that $f=0$ a.e. in $\Omega$ and leads to a contradiction. This completes the proof of Theorem~\ref{thm 1.4}.
\end{proof}

\textit{Acknowledgments.} I thank Antoine Lemenant for suggesting that I work on this problem, for his support and valuable comments on the manuscript. Also I am grateful to Antonin Chambolle for fruitful discussions. This work was partially supported by the project ANR-18-CE40-0013 SHAPO financed by the French Agence Nationale de la Recherche (ANR).
%%%%%%

%%%%%%
%%%%%%

\bibliography{bib1}

\begin{thebibliography}{10}

\bibitem{Potential}
David~R. Adams and Lars~Inge Hedberg.
\newblock {\em Function spaces and potential theory}, volume 314 of {\em
  Grundlehren der Mathematischen Wissenschaften [Fundamental Principles of
  Mathematical Sciences]}.
\newblock Springer-Verlag, Berlin, 1996.

\bibitem{BAGBY}
T.~Bagby.
\newblock Quasi topologies and rational approximation.
\newblock {\em Journal of Functional Analysis}, 10(3):259 -- 268, 1972.

\bibitem{Bucur}
Dorin Bucur and Paola Trebeschi.
\newblock Shape optimisation problems governed by nonlinear state equations.
\newblock {\em Proc. Roy. Soc. Edinburgh Sect. A}, 128(5):945--963, 1998.

\bibitem{partialregularity}
Bohdan Bulanyi.
\newblock Partial regularity for the optimal $p$-compliance problem with length
  penalization.
\newblock Preprint, 2021.

\bibitem{p-compl}
Bohdan Bulanyi and Antoine Lemenant.
\newblock Regularity for the planar optimal $p$-compliance problem.
\newblock {\em Preprint arXiv:1911.09240}, 2020.

\bibitem{Butazzo-Santambrogio}
Giuseppe Buttazzo and Filippo Santambrogio.
\newblock Asymptotical compliance optimization for connected networks.
\newblock {\em Netw. Heterog. Media}, 2(4):761--777, 2007.

\bibitem{Opt}
Antonin Chambolle, Jimmy Lamboley, Antoine Lemenant, and Eugene Stepanov.
\newblock Regularity for the optimal compliance problem with length
  penalization.
\newblock {\em SIAM J. Math. Anal.}, 49(2):1166--1224, 2017.

\bibitem{ABC}
Gianni Dal~Maso and Fran\c{c}ois Murat.
\newblock Asymptotic behaviour and correctors for {D}irichlet problems in
  perforated domains with homogeneous monotone operators.
\newblock {\em Ann. Scuola Norm. Sup. Pisa Cl. Sci. (4)}, 24, 1997.

\bibitem{PDE}
David Gilbarg and Neil S.Trudinger.
\newblock {\em Elliptic Partial Differential Equations of Second Order}, volume
  224 of {\em Grundlehren der mathematischen Wissenschaften}.
\newblock Springer-Verlag, Berlin, second edition, 2001.

\bibitem{Hedberg}
L.I. Hedberg.
\newblock Non-linear potentials and approximation in the mean by analytic
  functions.
\newblock {\em Math. Z.}, 129:299--319, 1972.

\bibitem{MR3063566}
Al-Hassem Nayam.
\newblock Asymptotics of an optimal compliance-network problem.
\newblock {\em Netw. Heterog. Media}, 8(2):573--589, 2013.

\bibitem{MR3195349}
Al-Hassem Nayam.
\newblock Constant in two-dimensional {$p$}-compliance-network problem.
\newblock {\em Netw. Heterog. Media}, 9(1):161--168, 2014.

\bibitem{Ziemer}
William~P. Ziemer.
\newblock {\em Weakly differentiable functions}, volume 120 of {\em Graduate
  Texts in Mathematics}.
\newblock Springer-Verlag, New York, 1989.
\newblock Sobolev spaces and functions of bounded variation.

\bibitem{sverak}
Vladimir Šverák.
\newblock On optimal shape design.
\newblock {\em J. Math. Pures Appl.}, 72:537--551, 1993.

\end{thebibliography}
\bibliographystyle{plain}

\end{document}